\documentclass[12pt,reqno]{amsart}
\usepackage{amsfonts,amssymb,latexsym,amsmath,amsthm,color,dsfont}
\usepackage{enumerate,enumitem,cite}
\usepackage{stmaryrd}
\usepackage{fullpage}
\usepackage{multirow}
\usepackage{makecell,array}
\usepackage{tablefootnote}
\usepackage[colorlinks,linkcolor=blue, citecolor=red,anchorcolor=blue, pdftex,unicode=true
]{hyperref}
\hypersetup{pdfencoding=auto}
\usepackage{bm}
\usepackage{tabularx}
\usepackage{graphicx, amsmath, amssymb, amsthm,anysize,float,booktabs,geometry}
\usepackage{pdflscape}
\usepackage{caption} 

 \usepackage[nobysame]{amsrefs}
 \def\MR#1{}                
\def\@bibmrnumber#1{}       
\def\@bib@mrreview#1{}      
\def\@bib@mathreviews#1{}   

\hypersetup{citecolor=red, linkcolor=blue, colorlinks=true}

\marginsize{2cm}{2cm}{2cm}{0cm}
\allowdisplaybreaks

\newcommand\F{\mathbb{F}}

\newcommand{\Aut}{\mathrm{Aut}}
\newcommand{\Sym}{\mathrm{Sym}}

\newcommand\PGO{\mathrm{PGO}}

\usepackage{cleveref}
\crefname{section}{§}{§§}
\Crefname{section}{§}{§§}

\newcommand{\ra}{\rangle}

\newcommand{\la}{\langle}

\newcommand{\cQ}{{\mathcal Q}}

\theoremstyle{plain}
\newtheorem{theorem}{Theorem}[section]

\newtheorem{lemma}[theorem]{Lemma}

\newtheorem{definition}[theorem]{Definition}

\newtheorem{conjecture}[theorem]{Conjecture}

\newtheorem{remark}[theorem]{Remark}
\numberwithin{equation}{section}

\def\<{\langle}
\def\>{\rangle}
\def\la{\langle}
\def\ra{\rangle}

\newcommand{\End}{\operatorname{End}}

\title{The complete classification of triply-transitive strongly regular graphs}
\author{Weicong Li, Hanlin Zou$^\ast$}
\thanks{$^\ast$Corresponding author}
\address{Weicong~Li, Department of Mathematics, School of Sciences, Great Bay University, Dongguan, China.}
\email{liweicong@gbu.edu.cn}

\address{Hanlin Zou, School of Mathematics and Statistics, Yunnan University, Kunming 650091, China}
\email{zouhanlin@ynu.edu.cn}

\begin{document}

\begin{abstract}

This paper completes the classification of triply-transitive strongly regular graphs, a program recently initiated by Herman, Maleki, and Razafimahatratra. By proving that the collinearity graph of the polar space $\mathcal{Q}^{-}(5,q)$ and the affine polar graph $\mathrm{VO}^{\varepsilon}_{2m}(2)$ are triply-transitive, we resolve the final open cases in the classification. The result is a definitive list of all strongly regular graphs that exhibit this exceptional form of local symmetry, characterized by the equality $T_{0,\omega}=T_{\omega}=\widetilde{T}_{\omega}$ of their Terwilliger algebras.


\medskip
\noindent{{\it Keywords\/}: strongly regular graphs, Terwilliger algebra, triple-transitivity, polar space, affine polar graph, association scheme}

\smallskip

\noindent {{\it MSC (2020)\/}: 05E30, 05C25, 20B25, 51E99}

\end{abstract}

\maketitle

\section{Introduction}\label{sec_intro}

The classification of highly symmetric combinatorial structures is a central pursuit in algebraic combinatorics. A particularly compelling class of such structures consists of strongly regular graphs that are \emph{triply-transitive} (see Definition \ref{def_tt}). These graphs exhibit an exceptional degree of symmetry, characterized by a perfect alignment between their local combinatorial data, their representation-theoretic algebras, and their global automorphism groups. Moreover, the investigation of these graphs provides crucial insights into a fundamental general problem: determining for which association schemes the Terwilliger algebra coincides with the centralizer algebra of the vertex stabilizer (see Subsection \ref{ss_asta}).

In a recent systematic study, Herman, Maleki, and Razafimahatratra~\cite{HMR25} initiated a program to classify all triply-transitive strongly regular graphs. Their work provided a nearly complete classification, identifying almost all known families that satisfy this stringent condition. Specifically, they proved that a triply-transitive strongly regular graph must be one of the following:
\begin{itemize}
    \item[(a)] a complete multipartite graph with $n$ parts of size $m$,
    \item[(b)] the $5$-cycle,
    \item[(c)] the McLaughlin graph,
    \item[(d)] the Higman-Sims graph,
    \item[(e)] the Peisert graph $P^*(9)$ (isomorphic to the Paley graph of order $9$),
    \item[(f)] an $n\times n$ grid for some $n \geq 2$,
    \item[(g)] the collinearity graph of the polar space $\mathcal{Q}^-(5,q)$,
    \item[(h)] the affine polar graph $\mathrm{VO}_{2m}^{\varepsilon}(2)$ for $m \geq 2$ and $\varepsilon = \pm 1$.
\end{itemize}
Furthermore, they confirmed that the graphs in families (a)--(f) are indeed triply-transitive. However, the status of the two infinite families, (g) and (h), remained unresolved. Based on analysis of small examples, they formulated the following conjectures for the two remaining infinite families.

\begin{conjecture}[{\cite[Conjecture 6.13]{HMR25}}]\label{conj1}
For any prime power $q$, the collinearity graph of the polar space $\mathcal{Q}^{-}(5,q)$ is triply-transitive.
\end{conjecture}
\begin{conjecture}[{\cite[Conjecture 6.17]{HMR25}}]\label{conj2}
For any integer $m \geq 2$ and $\varepsilon=\pm 1$, the affine polar graph $\mathrm{VO}^{\varepsilon}_{2m}(2)$ is triply-transitive.
\end{conjecture}

In this paper, we prove these two conjectures. By a detailed analysis of the orbits of specific point stabilizers, we show that both families satisfy the defining condition of triple transitivity. Our main result is thus the following complete classification theorem.

\begin{theorem}[Main Theorem]\label{thm_main}
A strongly regular graph is triply-transitive if and only if it is one of the graphs listed in (a)--(h) above.
\end{theorem}

The rest of this paper is organized as follows. In Section \ref{sec_prelim}, we review the fundamental concepts and properties of association schemes, Terwilliger algebras, and triply-transitive strongly regular graphs. Additionally, we provide preliminary results in finite fields that will be needed later. After that, we present the proofs of Conjectures \ref{conj1} and \ref{conj2}, and the proof of the main theorem in Section \ref{sec_main}.

\section{Preliminaries}\label{sec_prelim}

This section collects the necessary background material, beginning with the general theory of association schemes and then specializing to strongly regular graphs and the tools needed for our proofs.

\subsection{Association schemes and the Terwilliger algebra}\label{ss_asta}\

Let $\Omega$ be a finite nonempty set. A \emph{symmetric association scheme with $d$ classes} is a pair $(\Omega, \mathcal{R})$ where $\mathcal{R} = \{R_0, R_1, \ldots, R_d\}$ is a partition of $\Omega \times \Omega$ satisfying:
\begin{enumerate}[label=(\roman*)]
    \item $R_0 = \{(\omega,\omega) : \omega \in \Omega\}$;
    \item For each $i \in \{1, \ldots, d\}$, the relation $R_i$ is symmetric: $(\omega_1, \omega_2) \in R_i$ implies $(\omega_2, \omega_1) \in R_i$;
    \item For any $i, j, k \in \{0, 1, \ldots, d\}$, there exists a nonnegative integer $p_{ij}^k$ (the \emph{intersection number}) such that for any $(\omega_1, \omega_2) \in R_k$,
    \[
    |\{\omega_3 \in \Omega : (\omega_1, \omega_3) \in R_i \text{ and } (\omega_3, \omega_2) \in R_j\}| = p_{ij}^k.
    \]
\end{enumerate}

For each $i \in \{0, 1, \ldots, d\}$, let $A_i$ be the adjacency matrix of the relation $R_i$. Fix a vertex $\omega \in \Omega$. Define the diagonal matrix $E^*_{i,\omega}$ (the \emph{dual idempotent}) by
\begin{equation}\label{eq_Ei*}
E^*_{i,\omega} ({\alpha,\alpha}) = 
\begin{cases}
1 & \text{if } (\omega, \alpha) \in R_i, \\
0 & \text{otherwise}.
\end{cases}
\end{equation}
Let $\mathrm{Mat}_{|\Omega|}(\mathbb{C})$ be the full matrix algebra over $\mathbb{C}$ whose rows and columns are indexed by the elements of $\Omega$. The \emph{Terwilliger algebra} $T_\omega = T_\omega(\Omega, \mathcal{R})$ with respect to a vertex $\omega$ is the subalgebra of $\operatorname{Mat}_{|\Omega|}(\mathbb{C})$ generated by
\[
\{A_0, A_1, \ldots, A_d, E^*_{0,\omega}, E^*_{1,\omega}, \ldots, E^*_{d,\omega}\}.
\]
This algebra was introduced by Terwilliger \cites{Terwilliger1,Terwilliger2,Terwilliger3} as the \emph{subconstituent algebra} and provides a powerful tool for studying the local structure of association schemes.

Let $\Aut(\mathcal{R})$ denote the automorphism group of the association scheme, defined as the set of all permutations of $\Omega$ that preserve every relation in $\mathcal{R}$. For any subgroup $H \leq \Sym(\Omega)$, the \emph{centralizer algebra} is
\[
\End_H(\mathbb{C}^{|\Omega|}) = \{ X \in \operatorname{Mat}_{|\Omega|}(\mathbb{C}) : P_g X = X P_g, \forall g \in H \},
\]
where $P_g$ is the permutation matrix corresponding to $g$. 

For any $\omega \in \Omega$, we have the fundamental inclusion:
\begin{equation}\label{eq:fundamental-inclusion}
T_\omega \subseteq \End_{\Aut(\mathcal{R})_\omega}(\mathbb{C}^{|\Omega|}).
\end{equation}
Understanding when equality holds in \eqref{eq:fundamental-inclusion} is an important problem in algebraic combinatorics, see \cite{TFIL19}. This paper contributes to this direction by completely resolving the case of strongly regular graphs under the stronger condition of triple transitivity.

\subsection{Triply-transitive strongly regular graphs}\

A graph $\Gamma = (\Omega, E)$ is {\it strongly regular} with parameters $(v, k, \lambda, \mu)$ if it is a $k$-regular graph on $v$ vertices such that every pair of adjacent vertices has exactly $\lambda$ common neighbors, and every pair of distinct non-adjacent vertices has exactly $\mu$ common neighbors.

Every strongly regular graph $\Gamma = (\Omega,E)$ gives rise to a symmetric 2-class association scheme $(\Omega, \{R_0, R_1, R_2\})$, where $R_1$ and $R_2$ represent the edges and non-edges, respectively. Fixing a base vertex $\omega \in \Omega$, we obtain the canonical partition of the vertex set into \emph{subconstituents}:
\[
\Delta_0(\omega) = \{\omega\}, \quad \Delta_1(\omega) = \{\alpha\in \Omega : \omega\sim \alpha \}, \quad \Delta_2(\omega) = \{\alpha\in \Omega : \alpha \neq \omega, \omega\nsim \alpha\}.
\]
Let $A_0=I, A_1, A_2$ be the adjacency matrices of the relations $R_0, R_1, R_2$, respectively. The corresponding dual idempotents $E^*_{i,\omega}$ ($i=0,1,2$) are the diagonal projection matrices onto $\Delta_i(\omega)$ (see \eqref{eq_Ei*}). The Terwilliger algebra $T_\omega(\Gamma)$ is generated by $\{A_0, A_1, A_2, E^*_{0,\omega}, E^*_{1,\omega}, E^*_{2,\omega}\}$ and it contains the subspace
\[
T_{0,\omega} = \operatorname{Span}\{ E^*_{i,\omega} A_j E^*_{k,\omega} : i, j, k \in \{0,1,2\} \}.
\]

Let $G = \operatorname{Aut}(\Gamma)$ and let $G_\omega$ be the stabilizer of a vertex $\omega$. The centralizer algebra is denoted by
\[
\widetilde{T}_{\omega} = \operatorname{End}_{G_\omega}(\mathbb{C}^{|\Omega|}).
\]
The three algebras above satisfy the following natural chain of inclusions. 
\begin{equation}\label{eq_inclusion}
T_{0,\omega} \subset T_{\omega}\subset \widetilde{T}_{\omega}.
\end{equation}

\begin{definition}\label{def_tt}
A strongly regular graph is called \emph{triply-transitive} if it is vertex-transitive and
\[
T_{0,\omega} = T_\omega = \widetilde{T}_\omega,
\]
for any vertex $\omega$.
\end{definition}
\begin{remark}
In the above definition, the equality $T_{0,\omega} = T_{\omega}$ is equivalent to $\Gamma$ being \emph{triply-regular} (see \cite[Lemma 4]{Mun93}), a strong combinatorial property meaning that the number of vertices at prescribed distances from any triple depends only on the distances between the points in the triple, not on the specific triple chosen. The further equality $T_{\omega} = \widetilde{T}_{\omega}$ signifies that the algebra generated by the local combinatorial data is as large as it can possibly be, given the symmetries of the graph, perfectly capturing the symmetry imposed by the global automorphism group. Consequently, the classification of triply-transitive strongly regular graphs is the classification of those graphs which are, in a very precise sense, maximally symmetric from the viewpoint of the Terwilliger algebra.
\end{remark}

By \eqref{eq_inclusion} and Definition \ref{def_tt}, in order to show that a vertex-transitive strongly regular graph is triply-transitive, it suffices to show that $\dim(T_{0,\omega})=\dim(\widetilde{T})$. We now collect some tools for calculating the dimensions of $T_{0,\omega}$ and $\widetilde{T}$.

 Recall that a strongly regular graph $\Gamma$ is {\it primitive} if both $\Gamma$ and its complement are connected. Equivalently, if $\Gamma$ has parameters $(v,k,\lambda,\mu)$, then it is primitive precisely when $\lambda<k-1$ and $\mu<k$ (see \cite[1.1.3]{Srgs22}). The latter condition enables us to easily verify that the collinearity graph of $\mathcal{Q}^-(5,q)$ and the affine polar graph $\mathrm{VO}_{2m}^\varepsilon(2)$ are both primitive. To determine the dimension $\dim(T_{0,\omega})$ of these graphs, we will apply the following lemma.

\begin{lemma}[{\cite[Proposition 2.5(iii-iv)]{HMR25}}]\label{lem:dimT0}
Let $\Gamma$ be a primitive strongly regular graph. Then the following hold.

(1) If exactly one of $\Gamma$ and its complement contains triangles, then $\dim(T_{0,\omega}) = 14$.

(2) If both $\Gamma$ and its complement contain triangles, then $\dim(T_{0,\omega}) = 15$.
\end{lemma}


To compute $\dim(\widetilde{T}_{\omega})$, we use the \emph{block dimension decomposition} of $\widetilde{T}_{\omega}$, which is the matrix
\[
D(\widetilde{T}_{\omega}) = 
\begin{bmatrix}
d_{00} & d_{01} & d_{02} \\
d_{10} & d_{11} & d_{12} \\
d_{20} & d_{21} & d_{22}
\end{bmatrix},
 \text{ where } d_{ij} = \dim(E^*_{i,\omega} \widetilde{T}_{\omega} E^*_{j,\omega}).
\]
Since the idempotents $\{E^*_{i,\omega}\}$ are mutually orthogonal, the algebra $\widetilde{T}_{\omega}$ decomposes as a direct sum
\[
\widetilde{T}_{\omega} = \bigoplus_{0 \le i,j \le 2} E^*_{i,\omega} \widetilde{T}_{\omega} E^*_{j,\omega}.
\]
Therefore, we have 
\[\dim(\widetilde{T}_{\omega}) = \sum_{i,j} d_{ij}.\]
It was proved in \cite[Theorem 5.1]{HMR25} that triply-transitive strongly regular graphs are rank 3 graphs. For a rank $3$ graph, the structure of this decomposition is well-known, and its entries count certain orbitals of the point stabilizer. The following lemma is our primary tool for computing $\dim(\widetilde{T}_{\omega})$.

\begin{lemma}[{\cite[Theorem 3.2]{HMR25}}]\label{lem:block-decomp}
Let $\Gamma$ be a strongly regular graph with a rank $3$ automorphism group $G = \operatorname{Aut}(\Gamma)$. Fix $\omega \in \Omega$ and choose $\omega_1 \in \Delta_1(\omega)$ and $\omega_2 \in \Delta_2(\omega)$. Then the block dimension decomposition of $\widetilde{T}_{\omega}$ is:
\[
D(\widetilde{T}_{\omega}) = 
\begin{bmatrix}
1 & 1 & 1 \\
1 & r_1 & t \\
1 & t & r_2
\end{bmatrix},
\]
and consequently,
\[
\dim(\widetilde{T}_{\omega}) = 5 + r_1 + r_2 + 2t.
\]
Here, $r_i$ is the number of orbits of $G_{\omega} \cap G_{\omega_i}$ on $\Delta_i(\omega)$ ($i=1,2$), and $t$ is the number of orbits of $G_{\omega} \cap G_{\omega_2}$ on $\Delta_1(\omega)$.
\end{lemma}

In Section \ref{sec_main}, we will apply Lemma \ref{lem:block-decomp} to the collinearity graph of $\mathcal{Q}^-(5,q)$ and the affine polar graph $\mathrm{VO}_{2m}^\varepsilon(2)$. In both cases, we will show that 
\[D(\widetilde{T}_{\omega}) = \begin{bmatrix}1&1&1\\1&3&2\\1&2&3\end{bmatrix},\] so that $\dim(\widetilde{T}_{\omega}) = 15$. Since these graphs satisfy the conditions of Lemma \ref{lem:dimT0}, this proves that they are triply-transitive.

\subsection{A technical lemma in finite fields}\ 

In the final part of this section, we include a basic result on finite fields for later use. Let $q$ be a prime power and $\F_q$ the finite field of order $q$. Denote by $\mathbb{F}_{q}^{*}$ the multiplicative group of $\F_{q}$. When $q$ is odd, we define
\[
\square_{q}=\{x^{2}:x\in\mathbb{F}_{q}^{*}\},\quad \blacksquare_{q}=\mathbb{F}_{q}^{*}\setminus\square_{q}.
\]

\begin{lemma}\label{lem_tech}
Suppose that $q$ is odd. Let $\delta\in\blacksquare_{q}$. Define $f(x,y)=x^{2}-\delta y^{2}$. Then for any $\lambda\in\blacksquare_{q}$, there exist $a,c\in\mathbb{F}_{q}$ such that $f(a,c)=\lambda$. Moreover, let $g$ be a linear map on $\mathbb{F}_{q}^2$ defined by
\[
g:(x,y)\mapsto(ax+c\delta y,cx+ay).
\]
Then we have $f(g(x,y))=\lambda f(x,y)$.
\end{lemma}

\begin{proof}
If $q\equiv 1\pmod{4}$, we have $-1\in\square_{q}$. Then there exists $c\in\mathbb{F}_{q}$ such that $c^{2}=-\lambda\delta^{-1}$. Hence there exists a pair $(0,c) $ such that $f(0,c)=-\delta c^2=\lambda $. On the other hand, if $q\equiv 3\pmod{4}$, we have $-1\in\blacksquare_{q}$. Then $\lambda+\delta c^{2}\neq 0$ for any $c\in\mathbb{F}_{q}$. It follows that the set $\{\lambda+\delta c^{2}:c\in\mathbb{F}_{q}\}$ has cardinality $\frac{q+1}{2}$ in $\mathbb{F}_{q}^{*}$. It must contain a square in $\mathbb{F}_{q}^{*}$ since there are exactly $\frac{q-1}{2}$ nonsquares in $\mathbb{F}_{q}^{*}$. As a consequence, there exist $a,c\in\mathbb{F}_{q}$ such that $\lambda+\delta c^{2}=a^{2}$. This establishes the existence of a pair $(a,c)$ with $a^{2}-c^{2}\delta=\lambda$.

The remaining part follows from a direct computation that
\begin{align*}
f(g(x,y))&=f(ax+c\delta y,cx+ay )=(ax+c\delta y)^{2}-\delta(cx+ay)^{2}\\&=(a^{2}-\delta c^{2})x^{2}-\delta(a^{2}-\delta c^{2})y^{2}=\lambda f(x,y).
\end{align*}
This completes the proof.
\end{proof}

\section{Proof of the main theorem}\label{sec_main}
In this section, we prove Conjectures \ref{conj1} and \ref{conj2}, thereby completing the proof of the main classification theorem. We will prove that the two remaining families of strongly regular graphs are indeed triply-transitive. Before presenting the proofs, let us fix some notation that will be used frequently in this section.

Let $\Gamma$ be a strongly regular graph and let $G$ be its automorphism group. For a vertex $\omega$, let $\Delta_i=\Delta_i(\omega)$ for $i=0,1,2$. If $G$ is transitive, then $T_{0,\omega}\cong T_{0,\omega'}$, $T_{\omega}\cong T_{\omega'}$ and $\widetilde{T}_{\omega}\cong \widetilde{T}_{\omega'}$ for different vertices $\omega$ and $\omega'$, and therefore we will simplify the notation to $T_0, T$, and $\widetilde{T}$, respectively.

\subsection{Triply-transitivity of the collinearity graph of {$\mathcal{Q}^{-}(5,q)$}}\

We first recall the definition of the polar space $\mathcal{Q}^-(5,q)$ and its collinearity graph. Let $q$ be a prime power. Set $V=\mathbb{F}_{q}^{6}$ and view it as a 6-dimensional vector space over $\mathbb{F}_{q}$. Define a map $Q(\bm{x})$ on $V$ by
\begin{equation}
Q(\boldsymbol{x})=x_{1}x_{2}+x_{3}x_{4}+f(x_{5},x_{6}),
\end{equation}
where $f(x_{5},x_{6})$ is an irreducible polynomial over $\mathbb{F}_{q}$, which will be made precise in certain case later. Then $Q$ is a nondegenerate quadratic form of minus type, and it defines the polar space $\mathcal{Q}^{-}(5,q)$ which is also known as a finite classical generalized quadrangle of order $(q,q)$. We call a vector $\bm{x}\in V$ {\it singular} if $Q(\bm{x})=0$. The point set $\mathcal{P}$ of $\mathcal{Q}^{-}(5,q)$ consists of all 1-dimensional subspaces of $V$ generated by nonzero singular vectors, that is, 
\[
\mathcal{P}=\{\langle\boldsymbol{x}\rangle:\boldsymbol{x}\in V\setminus\{\bm{0}\} \mid Q(\boldsymbol{x})=0\}.
\]
The associated bilinear form $B:V\times V\to\mathbb{F}_{q}$ of $Q$ is defined by
\begin{equation}
B(\boldsymbol{x},\boldsymbol{y})=Q(\boldsymbol{x}+\boldsymbol{y})-Q(\boldsymbol{x})-Q(\boldsymbol{y}).
\end{equation}
Two points $\langle\boldsymbol{x}\rangle,\langle\boldsymbol{y}\rangle\in\mathcal{P}$ are collinear if and only if $B(\boldsymbol{x},\boldsymbol{y})=0$. We define
\[
\bm{x}^{\perp}:=\{\langle\boldsymbol{y}\rangle:\boldsymbol{y}\in V\mid B(\boldsymbol{x},\boldsymbol{y})=0\}.
\]

Let $\Gamma$ be the collinearity graph of $\mathcal{Q}^{-}(5,q)$. Then the vertices of $\Gamma$ are the points of $\mathcal{Q}^{-}(5,q)$, and two vertices are adjacent in $\Gamma$ if and only if they are collinear in $\mathcal{Q}^{-}(5,q)$. 

Let $\{\boldsymbol{e}_{1},\ldots,\boldsymbol{e}_{6}\}$ be the standard basis of $V$. Consider the following vertices of $\Gamma$:
\begin{align*}
\omega &=\langle\boldsymbol{e}_{1}\rangle=\langle(1,0,0,0,0,0)\rangle, \\
\omega_{1} &=\langle\boldsymbol{e}_{3}\rangle=\langle(0,0,1,0,0,0)\rangle, \\
\omega_{2} &=\langle\boldsymbol{e}_{2}\rangle=\langle(0,1,0,0,0,0)\rangle.
\end{align*}
By (3.1) and (3.2), we deduce that
\begin{align*}
\Delta_1&=\{\langle(x_{1},0,x_{3},x_{4},x_{5},x_{6})\rangle\in\mathcal{P}\}, \ \omega_1\in \Delta_1,\\
\Delta_2 &=\{\langle(x_{1},1,x_{3},x_{4},x_{5},x_{6})\rangle\in\mathcal{P}\},\  \omega_2\in \Delta_2.
\end{align*}

Let $G$ be the automorphism group of $\Gamma$. Then $G=\operatorname{P\Gamma O}_{6}^{-}(q)$ (see \cite[2.6.5]{Srgs22}). Denote by $G_{\omega}$ and $G_{\omega_{i}}$ the stabilizer of $\omega$ and $\omega_{i}$ in $G$, respectively. In what follows, we determine the orbits of the subgroups $G_{\omega}\cap G_{\omega_{i}}$ of $G$ on the sets $\Delta_{j}$ for $i,j\in\{1,2\}$. Denote by $H_{i}=G_{\omega}\cap G_{\omega_{i}}$ for $i=1,2$.

Set $W=\omega^{\perp}\cap\omega_{2}^{\perp}=\langle\boldsymbol{e}_{3},\boldsymbol{e}_{4},\boldsymbol{e}_{5},\boldsymbol{e}_{6}\rangle=\mathbb{F}_{q}^{4}$ and let $\widetilde{Q}=Q|_{W}$ denote the restriction of the quadratic form $Q$ to $W$. Then $(W,\widetilde{Q})$ defines a nondegenerate elliptic quadric $\mathcal{Q}^{\prime}=\mathcal{Q}^{-}(3,q)$.  Consider the following maps $\phi_{\lambda}$ on $V$ defined by
\begin{equation}\label{eq_phifunc}
\phi_{\lambda}(\boldsymbol{x})=(\lambda x_{1},\lambda^{-1}x_{2},x_{3},x_{4},x_{5},x_{6}),\quad \forall\lambda\in\mathbb{F}_{q}^{*}.
\end{equation}
It is clear that $\{\phi_{\lambda}:\lambda\in\mathbb{F}_{q}^{*}\}$ stabilizes $\omega,\omega_{1},\omega_{2}$, and thus it is a subgroup of order $q-1$ contained in $H_{1}\cap H_{2}$.

\begin{lemma}\label{lem_H2D1_1}
The subgroup $H_{2}$ has two orbits on $\Delta_1$.
\end{lemma}

\begin{proof}
Since $H_{2}=G_\omega\cap G_{\omega_2}$, it fixes $\omega$ and $\omega_{2}$, and thus fixes $W$. Let $Z_{1}$ be the set of all singular points in $W$ with respect to $\widetilde{Q}$, and set $Z_{2}=\Delta_1\setminus Z_{1}$. Then $H_{2}$ contains a subgroup $K=\operatorname{PGO}_{4}^{-}(q)$ that acts transitively on $Z_{1}$ (see \cite[Theorem 11.28]{Taylor}). We will show that $Z_{2}$ is also an orbit of $H_{2}$.

Given two vertices $\gamma_{1},\gamma_{2}\in Z_{2}$, there exist $\lambda,\mu\in\mathbb{F}_{q}^{*}$ and two singular vectors $\boldsymbol{u},\boldsymbol{v}\in W$ such that $\gamma_{1}=\langle\lambda\boldsymbol{e}_{1}+\boldsymbol{u}\rangle$ and $\gamma_{2}=\langle\mu\boldsymbol{e}_{1}+\boldsymbol{v}\rangle$. By Witt's Theorem (see \cite[Theorem 7.4]{Taylor}), there exists $\varphi\in K$ such that $\varphi(\langle\boldsymbol{u}\rangle)=\langle\boldsymbol{v}\rangle$. Assume that $\varphi(\langle\boldsymbol{e}_{1}+\boldsymbol{u}\rangle)=\langle\alpha\boldsymbol{e}_{1}+\boldsymbol{v}\rangle$ for some $\alpha\in\mathbb{F}_{q}^{*}$. Then
\[
\phi_{\mu\alpha^{-1}}\circ\varphi\circ\phi_{\lambda^{-1}}(\gamma_{1})=\phi_{\mu\alpha^{-1}}\circ\varphi(\langle\boldsymbol{e}_{1}+\boldsymbol{u}\rangle)=\phi_{\mu\alpha^{-1}}(\langle\alpha\boldsymbol{e}_{1}+\boldsymbol{v}\rangle)=\gamma_{2},
\]
where $\phi_{\mu\alpha^{-1}},\phi_{\lambda^{-1}}$ are defined as in \eqref{eq_phifunc}. This shows that $H_{2}$ acts transitively on $Z_{2}$. Therefore, $H_{2}$ has two orbits on $\Delta_1$.
\end{proof}

\begin{lemma}\label{lem_H2D2_1}
The subgroup $H_{2}$ has three orbits on $\Delta_2$.
\end{lemma}

\begin{proof}
Recall that $\widetilde{Q}=Q|_{W}$. Then for any vertex $\langle(x_{1},1,x_{3},x_{4},x_{5},x_{6})\rangle\in\Delta_2$, we have $x_{1}=-\widetilde{Q}(x_{3},x_{4},x_{5},x_{6})$. Consider the following three subsets of $\Delta_2$:
\begin{align*}
X_{1} &=\{\langle(0,1,0,0,0,0)\rangle\}=\{\omega_{2}\}, \\
X_{2} &=\{\langle(0,1,x_{3},x_{4},x_{5},x_{6})\rangle\in\Delta_2:\widetilde{Q}((x_{3},x_{4},x_{5},x_{6}))=0\}\setminus\{\omega_{2}\}, \\
X_{3} &=\{\langle(x_{1},1,x_{3},x_{4},x_{5},x_{6})\rangle\in\Delta_2:\widetilde{Q}((x_{3},x_{4},x_{5},x_{6}))=-x_{1}\neq 0\}.
\end{align*}
Note that $(W,\widetilde{Q})$ is a non-degenerate orthogonal space of minus type, which is stabilized by $H_2$. Thus $H_{2}$ contains a subgroup that acts transitively on $X_{2}$. It follows that $X_{1}$ and $X_{2}$ are two $H_{2}$-orbits.

Next, we consider the action of $H_{2}$ on $X_{3}$. First, we assume that $q$ is even. Let $\boldsymbol{u},\boldsymbol{v}\in W$ such that $\widetilde{Q}(\boldsymbol{u})\neq 0$ and $\widetilde{Q}(\boldsymbol{v})\neq 0$. Replacing $v$ by $\lambda v$ with $\lambda\in\mathbb{F}_{q}^{*}$ if necessary, we have $\widetilde{Q}(\boldsymbol{v})=\widetilde{Q}(\boldsymbol{u})$. By \cite[Lemma 2.5.10]{KL1990}, there exists an isometry $\widetilde{\theta}$ of $(W,\widetilde{Q})$ that maps $\boldsymbol{u}$ to $\boldsymbol{v}$. In fact, we have $\widetilde{Q}(\lambda\boldsymbol{v})=\lambda^{2}\widetilde{Q}(\boldsymbol{v})=\widetilde{Q}(\boldsymbol{u})$. By Witt's Theorem (see \cite[Theorem 7.4]{Taylor}), $\widetilde{\theta}$ can be extended to an isometry $\theta$ of $(V,Q)$ such that $\theta(\boldsymbol{u})=\boldsymbol{v}$. This shows that $X_{3}$ is an orbit of $H_{2}$ when $q$ is even.

Finally, we treat the case where $q$ is odd. The set $X_{3}$ can be partitioned into two subsets, $S_{1}$ and $S_{2}$, where
\begin{align*}
S_{1} &=\{\langle\boldsymbol{x}\rangle\in X_{3}:\widetilde{Q}((x_{3},x_{4},x_{5},x_{6}))\in\square_{q}\}, \\
S_{2} &=\{\langle\boldsymbol{x}\rangle\in X_{3}:\widetilde{Q}((x_{3},x_{4},x_{5},x_{6}))\in\blacksquare_{q}\}.
\end{align*}
Since $H_{2}$ contains an isometry group $K\cong\mathrm{PGO}_{4}^{-}(q)$ that stabilizes the subspace $(W,\widetilde{Q})$, we deduce from \cite[Lemma 2.10.5]{KL1990} that $K$ acts transitively on both $S_{1}$ and $S_{2}$. We will show that there exists a $g\in H_{2}$ that maps $S_{1}$ to $S_{2}$. Let $\delta\in\blacksquare_{q}$ and assume without loss of generality that $f(x_{5},x_{6})=x_{5}^{2}-\delta x_{6}^{2}$, which is irreducible over $\mathbb{F}_{q}$. Then we choose the quadratic form defining $\mathcal{Q}^{-}(5,q)$ to be
\[
Q(\boldsymbol{x})=x_{1}x_{2}+x_{3}x_{4}+x_{5}^{2}-\delta x_{6}^{2}.
\]
By Lemma \ref{lem_tech}, there exists a similarity $\rho$ acting on $V$ given by
\[
\rho(\boldsymbol{x})=(\alpha x_{1},x_{2},\alpha x_{3},x_{4},ax_{5}+\delta cx_{6},cx_{5}+ax_{6}),
\]
where $\alpha\in\blacksquare_{q}$ and $\alpha=a^{2}-\delta c^{2}$ for some $a,c\in\mathbb{F}_{q}$. It follows that $Q(\rho(\boldsymbol{x}))=\alpha Q(\boldsymbol{x})$ for $\boldsymbol{x}\in V$, and then $S_{1}$ and $S_{2}$ are fused by $\rho$. Therefore, $X_{3}$ is an orbit. This completes the proof.
\end{proof}

\begin{lemma}\label{lem_H1D1_1}
The subgroup $H_{1}$ has three orbits on $\Delta_1$.
\end{lemma}

\begin{proof}
Let $\Gamma_{1}$ be the set of neighbors of $\omega_{1}$ in $\Delta_1$ and let $\Gamma_{2}=\Delta_1\setminus(\{\omega_{1}\}\cup\Gamma_{1})$. It is clear that $\{\omega_{1}\}$ is an orbit of $H_{1}$. Next, we show that $\Gamma_{1}$ is also an orbit. For two vertices $\gamma_{1},\gamma_{2}\in\Gamma_{1}$, since they are adjacent to $\omega$, we can write $\gamma_{1}=\langle\lambda\boldsymbol{e_{1}}+\boldsymbol{e_{3}}\rangle$ and $\gamma_{2}=\langle\mu\boldsymbol{e_{1}}+\boldsymbol{e_{3}}\rangle$ for some $\lambda,\mu\in\mathbb{F}_{q}^{*}$. Thus $\phi_{\mu\lambda^{-1}}(\gamma_{1})=\gamma_{2}$. Here, the map $\phi_{\mu\lambda^{-1}}$ is defined by \eqref{eq_phifunc}. This shows that $H_{1}$ is transitive on $\Gamma_{1}$.

Finally, we show that $\Gamma_{2}$ is also an orbit. Recall that $H_{2}$ is the stabilizer of the subspace $(W,\widetilde{Q})$, and so $H_{2}\cong\mathrm{P\Gamma O}_{4}^{-}(q)$. Let $\mathcal{T}$ be the set of isotropic vectors in $W$. Since $H_{2}$ acts $2$-transitively on $\mathcal{T}$ (see \cite[Theorem 11.28]{Taylor}), we have $H_{2}\cap G_{\omega_{1}}=H_{1}\cap G_{\omega_{2}}$, which acts transitively on $\mathcal{T}\setminus\{\omega_{1}\}$. Note that
\[
\Gamma_{2}=\{\langle\lambda\boldsymbol{e_{1}}+\boldsymbol{u}\rangle:\lambda\in\mathbb{F}_{q},\langle\boldsymbol{u}\rangle\in\mathcal{T}\}.
\]
Thus, in order to show that $H_{1}$ is transitive on $\Gamma_{2}$, it suffices to show that for any $\langle\boldsymbol{u}\rangle\in \mathcal{T}\setminus\{\omega_{1}\}$ and $\lambda\in\mathbb{F}_{q}^{*}$, there is an element of $H_{1}$ mapping $\langle \boldsymbol{u} \rangle$ to $\langle\lambda\boldsymbol{e_{1}}+\boldsymbol{u}\rangle$. By a direct computation, we have
\[
\mathcal{T}\setminus\{\omega_{1}\}=\{\langle\boldsymbol{e_{4}}\rangle\}\cup\{\langle(0,0,1,-f(x_{5},x_{6}),x_{5},x_{6})\rangle:x_{5},x_{6}\in\mathbb{F}_{q} \mid (x_{5},x_{6})\neq(0,0)\}.
\]
Take $\lambda\in\mathbb{F}_{q}^{*}$ and $\boldsymbol{u}=(0,0,u_{3},u_{4},u_{5},u_{6})$ so that $\langle\boldsymbol{u}\rangle\in\mathcal{T}\setminus\{\omega_{1}\}$. Then $u_{4}\neq 0$ since $f$ is irreducible over $\F_q$. Consider the following map $\theta_\lambda$ on $V$:
\[
\theta_\lambda(\boldsymbol{x})=(x_{1}+\lambda u_{4}^{-1}x_{4},x_{2},x_{3}-\lambda u_{4}^{-1}x_{2},x_{4},x_{5},x_{6}), \forall\, \boldsymbol{x}\in V.
\]
It is easy to verify that $Q(\theta_\lambda(\bm{x}))=Q(\bm{x})$, so $\theta_\lambda$ is an isometry of $(V, Q)$ for all $\lambda\in \F_q^*$. Moreover, we have $\theta_\lambda(\boldsymbol{e_{1}})=\boldsymbol{e_{1}}$, $\theta_\lambda(\boldsymbol{e_{3}})=\boldsymbol{e_{3}}$ and $\theta_\lambda(\boldsymbol{u})=\lambda\boldsymbol{e_{1}}+\boldsymbol{u}$. So these maps and their inverses are all in $H_{1}$. Recall that $H_1\cap G_{\omega_2}$ acts transitively on $\mathcal{T}\setminus \{\omega_1\}$, and it implies that $H_1$ acts transitively on $\Gamma_2$. Therefore, $H_{1}$ has three orbits on $\Delta_1$. This completes the proof.
\end{proof}



\begin{theorem}\label{thm_QTT}
For any prime power $q$, the collinearity graph of $\mathcal{Q}^-(5,q)$ is triply-transitive.
\end{theorem}
\begin{proof}
It is well-known that $G$ is transitive on the vertex set of $\Gamma$ (see \cite[Theorem 11.30]{Taylor}), and $\Gamma$ is strongly regular with parameters $((q+1)(q^3+1), q(q^2+1), q-1, q^2+1)$ (see \cite[2.6.3]{Srgs22}). By \cite[1.1.3]{Srgs22}, we see that $\Gamma$ is primitive. 
Moreover, both $\Gamma$ and its complement contain triangles. For example, the points $\la {\bf e_1}\ra,\la {\bf e_3}\ra,\la {\bf e_1}+{\bf e_3}\ra$ are mutually collinear in $\cQ^-(5,q)$, forming a triangle in $\Gamma$. And the points $\la {\bf e_1} \ra,\la {\bf e_2}\ra, \la {\bf e_1}+{\bf e_2}-{\bf e_3}+{\bf e_4}\ra$  are pairwise noncollinear, forming a triangle in the complement graph. By Lemma \ref{lem:dimT0}, we have $\dim(T_0)=15$. 
By Lemmas \ref{lem:block-decomp}, \ref{lem_H2D1_1}, \ref{lem_H2D2_1} and \ref{lem_H1D1_1}, we obtain that the block dimension decomposition of $\widetilde{T}$ is $\begin{bmatrix}1&1&1\\1&3&2\\1&2&3\end{bmatrix}$ and $\dim(\widetilde{T})=15$. Then we conclude that $T_0=T=\widetilde{T}$ by comparing their dimensions. Therefore, $\Gamma$ is triply-transitive.
\end{proof}

\subsection{Triply transitivity of the affine polar graph {$\mathrm{VO}_{2m}^{\varepsilon}(2)$}}\

In this subsection, we always assume that $m\geq 2$ is an integer and $\varepsilon=\pm 1$. We first recall the definition of the graph $\mathrm{VO}_{2m}^{\varepsilon}(2)$. Let $V=\mathbb{F}_{2}^{2m}$ and define a map $Q(\bm{x})$ on $V$ by
\[
Q(\boldsymbol{x})=x_{1}x_{2}+\cdots+x_{2m-3}x_{2m-2}+f(x_{2m-1},x_{2m}),
\]
where 
\[f(x_{2m-1},x_{2m})=\begin{cases}
x_{2m-1}x_{2m},&\text{ if }\varepsilon=1,\\
x_{2m-1}^{2}+x_{2m-1}x_{2m}+x_{2m}^{2},&\text{ if }\varepsilon=-1.
\end{cases}\]
Then $Q$ is a nondegenerate quadratic form on $V$ of type $\varepsilon$. We call a vector $\bm{x}\in V$ {\it singular} if $Q(\bm{x})=0$. 

The associated bilinear form $B:V\times V\to\mathbb{F}_{2}$ of $Q$ is defined by
\begin{equation}
B(\boldsymbol{x},\boldsymbol{y})=Q(\boldsymbol{x}+\boldsymbol{y})-Q(\boldsymbol{x})-Q(\boldsymbol{y}).
\end{equation}
Given two vectors $\bm{x},\bm{y}\in V$, we say that they are {\it orthogonal} if $B(\bm{x},\bm{y})=0$. 
For a vector $\bm{x}\in V$, we define its orthogonal complement as
\[
\bm{x}^{\perp}:=\{\langle\boldsymbol{y}\rangle:\boldsymbol{y}\in V\mid B(\boldsymbol{x},\boldsymbol{y})=0\}.
\]
The affine polar graph $\mathrm{VO}_{2m}^{\varepsilon}(2)$ has $V$ as its vertex set, and two vectors $\bm{u}$ and $\bm{v}$ are adjacent if $Q(\bm{u}-\bm{v})=0$.

Let $\{\bm{e}_1,\ldots,\bm{e}_{2m}\}$ be the standard basis for $V$. We fix the following notation:
\begin{align*}
\omega &=\bm{0}=(0,0,0,\ldots,0), \\
\omega_{1} &=\bm{e}_1=(1,0,0,\ldots,0), \\
\omega_{2} &=\bm{e}_1+\bm{e}_2=(1,1,0,\ldots,0).
\end{align*}
Then $\Delta_1$ consists of all nonzero singular vectors in $V$ and $\Delta_2$ consists of all non-singular vectors.

Let $G$ be the automorphism group of $\mathrm{VO}_{2m}^{\varepsilon}(2)$. Then all the isometries of $(V,Q)$ belong to $G$ and $G=\mathrm{AGO}_{2m}^\varepsilon(2)\cong T_V\rtimes \mathrm{GO}_{2m}^\varepsilon(2)$,  where $T_V=\{x\mapsto x+v: v\in V\}$, cf. \cite[Section 3.1]{DPSS25}.  Set $H_i = G_{\omega} \cap G_{\omega_i}$ for $i=1,2$. 

It is known that $\mathrm{VO}_{2m}^{\varepsilon}(2)$ is triply-regular \cite[Proposition 3.6.1]{Srgs22}. Furthermore, with respect to the vertex $\omega$, its first and second subconstituents are the graphs $\Gamma(\mathcal{Q}^{\varepsilon}(2m-1,2))$ and $\mathrm{NO}_{2m}^{\varepsilon}(2)$, respectively (see \cite{BS1990, Srgs22}); these are the graphs on the nonzero singular and nonsingular vectors of $V$, with adjacency defined by orthogonality. For $m \geq 3$, the automorphism groups of both $\Gamma(\mathcal{Q}^{\varepsilon}(2m-1,2))$ and $\mathrm{NO}_{2m}^{\varepsilon}(2)$ are rank 3. As established in \cite[Section 6]{HMR25}, this implies that 
the block dimension decomposition of $\widetilde{T}$ takes the form:
\begin{equation}\label{eq_BDD2}
D(\widetilde{T})=\begin{bmatrix}
1 & 1 & 1 \\
1 & 3 & t \\
1 & t & 3
\end{bmatrix},
\end{equation}
where $t$ is the number of orbits of $H_2$ on $\Delta_1$. Consequently, $\dim(\widetilde{T}) = 11 + 2t$, and determining this dimension reduces to computing the parameter $t$.




\begin{lemma}\label{lem_H2D1_2}
Let $m\geq 3$. The subgroup $H_{2}$ has two orbits in $\Delta_1$.
\end{lemma}

\begin{proof}
Let $\Gamma_{1}$ be the set of common neighbors of $\omega$ and $\omega_{2}$, let $\Gamma_{2}=\Delta_1\setminus\Gamma_{1}$. We will show that $H_{2}$ is transitive on both $\Gamma_{1}$ and $\Gamma_{2}$. Let $\widetilde{Q}$ be the restriction of $Q$ on $W=\langle\boldsymbol{e_{3}},\ldots,\bm{e}_{2m}\rangle$.

We first consider the action of $H_{2}$ on $\Gamma_{1}$. We have
\[
\Gamma_{1}=\{(1,0,\ldots,0)\}\cup\{(0,1,0,\ldots,0)\}\cup Z_{1}\cup Z_{2},
\]
where
\begin{align*}
Z_{1} &=\{(1,0,x_{3},\ldots,x_{2m})\in V \mid (x_{3},\ldots,x_{2m})\neq(0,\ldots,0),\widetilde{Q}((x_{3},\ldots,x_{2m}))=0\}, \\
Z_{2} &=\{(0,1,x_{3},\ldots,x_{2m})\in V \mid (x_{3},\ldots,x_{2m})\neq(0,\ldots,0),\widetilde{Q}((x_{3},\ldots,x_{2m}))=0\}.
\end{align*}

Since $H_{2}$ fixes $\omega$ and $\omega_{2}$, it contains a subgroup $K$ fixing both $\bm{e}_1$ and $\bm{e}_2$. Thus $K$ fixes $\<\bm{e}_1\>^\perp\cap \<\bm{e}_2\>^\perp=W$. It follows that $K$ contains a subgroup isomorphic to $\PGO_{2m-2}^{\varepsilon}(2)$, which is transitive on both $Z_1$ and $Z_2$ (see \cite[Theorem 2.10.5]{KL1990}). Now, we consider the map $\theta$ on $V$ defined by
\[
\theta(\boldsymbol{x})=(x_{1}+x_{4},x_{2}+x_{4},x_{1}+x_{2}+x_{3}+x_{4},x_{4},\ldots,x_{2m}).
\]
Since $Q(\theta(\bm{x}))=Q(\bm{x})$, $\theta$ is an isometry of $(V,Q)$ that fixes both $\omega$ and $\omega_{2}$. Thus $\theta\in H_2$. Moreover, we have $\theta((1,0,\ldots,0))=(1,0,1,0,\ldots,0)\in Z_{1}$ and $\theta((0,1,\ldots,0))=(0,1,1,0,\ldots,0)\in Z_{2}$. Thus $H_2$ is transitive on both $Z_1':=Z_{1}\cup\{(1,0,\ldots,0)\}$ and $Z_2':=Z_{2}\cup\{(0,1,0,\ldots,0)\}$. Finally, we show that $Z_1'$ and $Z_2'$ merge into one orbit by considering the map $\rho$ on $V$ defined by
\[
\rho(\boldsymbol{x})=(x_{2},x_{1},x_{3},\ldots,x_{2m}).
\]
It is clear that $\rho$ is an isometry of $(V,Q)$ that fixes $\omega$ and $\omega_2$, and so $\rho\in H_{2}$. Moreover, we have $\rho(Z^{\prime}_{1})=Z^{\prime}_{2}$. Therefore, $H_{2}$ is transitive on $\Gamma_{1}$.

In the remaining part of the proof, we consider the action of $H_{2}$ on $\Gamma_{2}$. 
We have $\Gamma_{2}=S_{1}\cup S_{2}$, where
\begin{align*}
S_{1} &=\{(0,0,x_{3},\ldots,x_{2m})\in V \mid (x_{3},\ldots,x_{2m})\neq(0,\ldots,0),\widetilde{Q}((x_{3},\ldots,x_{2m}))=0\}, \\
S_{2} &=\{(1,1,x_{3},\ldots,x_{2m})\in V \mid \widetilde{Q}((x_{3},\ldots,x_{2m}))=1\}.
\end{align*}
By a similar argument as in the first part, we see that $H_2$ contains a subgroup that is transitive on both $S_1$ and $S_2$. To complete the proof, we show that $S_{1}$ and $S_{2}$ can be fused by an isometry of $(V,Q)$. Take $\bm{u}=(0,0,1,0,\ldots,0)\in S_{1}$ and $\bm{v}=(1,1,1,1,0,\ldots,0)\in S_{2}$. Define a map $\phi$ on $V$ by
\[
\phi(\boldsymbol{x})=(x_{1}+x_{3},x_{2}+x_{3},x_{3},x_1+x_2+x_3+x_4,x_{5},\ldots,x_{2m}).
\]
It is straightforward to see that $\phi$ is an isometry of $(V,Q)$, and so $\phi\in G$. Moreover, we have $\phi(\omega)=\omega$, $\phi(\omega_{2})=\omega_{2}$ and $\phi(\bm{u})=\bm{v}$. Thus $H_{2}$ is transitive on $S_{1}\cup S_{2}=\Gamma_{2}$. This completes the proof.
\end{proof}



\begin{theorem}\label{thm_VOTT}
For any integer $m\geq 2$ and $\varepsilon=\pm 1$, the affine polar graph $\mathrm{VO}_{2m}^\varepsilon(2)$ is triply-transitive. 
\end{theorem}
\begin{proof}
Write $\Gamma=\mathrm{VO}_{2m}^\varepsilon(2)$. 
It is well-known that $G$ is transitive on the vertex set of $\Gamma$ (see \cite[Theorem 11.30]{Taylor}), and $\Gamma$ has parameters $(2^{2m}, (2^m-\varepsilon)(2^{m-1}+\varepsilon), 2(2^{m-1}-\varepsilon)(2^{m-2}+\varepsilon), 2^{m-1}(2^{m-1}+\varepsilon))$ (see \cite[3.3.1]{Srgs22}). By \cite[1.1.3]{Srgs22}, we see that $\Gamma$ is primitive. The rest of the proof is divided into two cases.

We first treat the case where $(m,\varepsilon)=(2,-1)$. The graph $\Gamma$ has parameters $(16,5,0,2)$ and so it does not contain a triangle. But the complement of $\Gamma$ contains a triangle, for example, the one induced by the vertices $\bm{0},\bm{e}_1+\bm{e}_2$ and $\bm{e}_2+\bm{e}_3+\bm{e}_4$. By Lemma \ref{lem:dimT0}, we have $\dim(T_0)=14$. One can check by Magma \cite{Magma} that $\dim(\widetilde{T})=14$. Therefore, $T_0=T=\widetilde{T}$ by comparing their dimensions, and consequently, we have $\Gamma$ is triply-transitive. 

Next, we assume that $(m,\varepsilon)\neq (2,-1)$. In this case, both $\Gamma$ and its complement contain triangles. For example, the vertices $\bm{0},\bm{e}_1$ and $\bm{e}_3$ form a triangle in $\Gamma$; and the vertices $\bm{0}, \bm{e}_1+\bm{e_2}$ and $\bm{e}_2+\bm{e}_3+\bm{e}_4$ form a triangle in the complement of $\Gamma$.
By Lemma \ref{lem:dimT0}, we have $\dim(T_0)=15$. On the other hand, if $(m,\varepsilon)=(2,+1)$, then it can be verified by Magma \cite{Magma} that $\dim(\widetilde{T})=15$; if $m\geq 3$, then we deduce from Lemmas \ref{lem:block-decomp} and \ref{lem_H2D1_2}, and Equation \eqref{eq_BDD2} that $\dim(\widetilde{T})=15$. We then conclude that $T_0=T=\widetilde{T}$ by comparing their dimensions. Therefore, $\Gamma$ is triply-transitive.
\end{proof}

\subsection{Proof of the main classification theorem}

With the triple transitivity of the two unresolved families established, we now present the proof of our main result.

\begin{proof}[Proof of Theorem \ref{thm_main}]
By \cite[Theorems 1.3 and 1.4]{HMR25}, any triply-transitive strongly regular graph must belong to one of the eight families listed in Section \ref{sec_intro} as (a)--(h).
Conversely, it was already proved in~\cite{HMR25} that the graphs in families (a)--(f) are triply-transitive. We have proved in Theorems \ref{thm_QTT} and \ref{thm_VOTT} that the graphs in families (g) and (h) are triply-transitive. This completes the classification.
\end{proof}

\section*{Acknowledgement} 
Weicong Li acknowledges the support of the National Natural Science Foundation of China Grant No. 12301422. Hanlin Zou acknowledges the support of the National Natural Science Foundation of China Grant No. 12461061.

\bibliographystyle{plain}

\end{document}